\documentclass{article}

\usepackage{graphicx}
\usepackage{indentfirst}
\usepackage{amsmath,amsfonts,amsthm,amssymb,amsrefs}
\usepackage{mathrsfs,stmaryrd}
\usepackage{amscd}
\usepackage{hyperref}

\usepackage{relsize}

\def\co{\colon\thinspace}
\DeclareMathAlphabet{\mathsfsl}{OT1}{cmss}{m}{sl}

\newcommand{\spinc}{{\mathrm{Spin}^c}}

\newtheorem{thm}{Theorem}[section]
\newtheorem{lem}[thm]{Lemma}
\newtheorem{cor}[thm]{Corollary}
\newtheorem{prop}[thm]{Proposition}

\theoremstyle{definition}
\newtheorem{defn}[thm]{Definition}

\newtheorem{rem}[thm]{Remark}

\begin{document}

\title{Exceptional surgeries on hyperbolic fibered knots}

\author{{\Large Yi NI}\\{\normalsize Department of Mathematics, Caltech, MC 253-37}\\
{\normalsize 1200 E California Blvd, Pasadena, CA
91125}\\{\small\it Emai\/l\/:\quad\rm yini@caltech.edu}}

\date{}
\maketitle

\begin{abstract}
Let $K\subset S^3$ be a hyperbolic fibered knot such that $S^3_{p/q}(K)$, the $\frac pq$--surgery on $K$, is non-hyperbolic. We prove that if the monodromy of $K$ is right-veering, then
$0\le\frac pq\le 4g(K)$. The upper bound $4g(K)$ cannot be attained if $S^3_{p/q}(K)$ is a small Seifert fibered L-space. If the monodromy of $K$ is neither right-veering nor left-veering, then $|q|\le3$. 

As a corollary,
for any given positive torus knot $T$, if $p/q\ge4g(T)+4$, then $p/q$ is a characterizing slope. This improves earlier bounds of Ni--Zhang and McCoy. We also prove that some finite/cyclic slopes are characterizing. More precisely,
$14$ is characterizing for $T_{4,3}$, $17$ is characterizing for $T_{5,3}$, and $4n+1$ is characterizing for $T_{2n+1,2}$ except when $n=5$. By a recent theorem of Tange, this shows that $T_{2n+1,2}$ is the only knot in $S^3$ admitting a lens space surgery while the Alexander polynomial has the form $t^n-t^{n-1}+t^{n-2}+\text{lower order terms}$.

In the appendix, we prove that if the rank of the second term of the knot Floer homology of a fibered knot is $1$, then the monodromy is either right-veering or left-veering.
\end{abstract}

\section{Introduction}

Given a knot $K\subset S^3$, let $S^3_{p/q}(K)$ be the manifold obtained by $\frac pq$--surgery on $K$. The main theorem in this paper gives a constraint on exceptional surgeries on hyperbolic fibered knots.

\begin{thm}\label{thm:NotSFS}
Let $K\subset S^3$ be a hyperbolic fibered knot such that $S^3_{p/q}(K)$ is non-hyperbolic. 
\newline(1) If the monodromy of $K$ is right-veering, then $0\le\frac pq\le 4g(K)$. Moreover, if $S^3_{p/q}(K)$ is a small Seifert fibered L-space, then $\frac pq\ne4g(K)$. 
\newline(2) If the monodromy of $K$ is neither right-veering nor left-veering, then $|q|\le2$.
\end{thm}



Right-veering diffeomorphisms were defined in \cite{HKM}. (This concept was originally conceived by Gabai.)
Let $F$ be a compact oriented surface with boundary, $a,b\subset F$ be two properly embedded arcs with $a(0)=b(0)=x\in\partial F$. We isotope $a,b$ with endpoints fixed, so that $|a\cap b|$ is minimal.
We say $b$ is {\it to the right of $a$} at $x$, if either $b$ is isotopic to $a$ with endpoints fixed, or $(b\cap U)\setminus\{x\}$ lies in the ``right'' component of $U\setminus a$, where $U\subset F$ is a small neighborhood of $x$. See Figure~\ref{fig:ToRight}. A diffeomorphism $\phi\co F\to F$ with $\phi|_{\partial F}=\mathrm{id}$ is {\it right-veering} if for every $x\in\partial F$ and every properly embedded arc $a\subset F$ with $x\in a$, the image $\phi(a)$ is to the right of $a$ at $x$. Similarly, we can define {\it left-veering} diffeomorphisms.

\begin{figure}[ht]
\begin{picture}(340,70)
\put(100,6){\scalebox{0.5}{\includegraphics*
{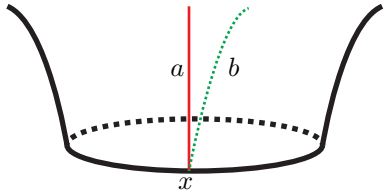}}}

\put(162,43){$a$}

\put(184,43){$b$}

\put(165,0){$x$}

\end{picture}
\caption{\label{fig:ToRight}The arc $b$ is to the right of $a$ at $x$.}
\end{figure}

A rational homology sphere $Y$ is an {\it L-space} if $\mathrm{rank}\widehat{HF}(Y)=|H_1(Y)|$. A knot $K\subset S^3$ is an {\it L-space knot}, if there exists a slope $\frac pq>0$ such that $S^3_{p/q}(K)$ is an L-space. 

\begin{rem}
Theorem~\ref{thm:NotSFS} was more or less known to experts. The point we want to make is that this theorem becomes very useful when combined with Heegaard Floer homology. First of all, knot Floer homology detects whether a knot is fibered \cite{Gh,NiFibred}. Secondly, the right-veering condition can often be checked using Heegaard Floer homology. Honda--Kazez--Mati\'c \cite{HKM} proved that if an open book supports a tight contact structure, then the corresponding monodromy is right-veering. It is often possible to get tightness from Heegaard Floer homology \cite{OSzContact}. For example, when $K$ is an L-space knot, then $K$ is fibered, and the open book with binding $K$ supports the unique tight contact structure on $S^3$, thus the monodromy of $K$ is right-veering. In the appendix, we will prove that if $\mathrm{rank}\widehat{HFK}(S^3,K,g(K)-1)=1$ for a fibered knot $K$, then its monodromy is either right-veering or left-veering.
\end{rem}

Theorem~\ref{thm:NotSFS} can be used to get bounds on characterizing slopes for torus knots.
A slope $\frac pq\in \mathbb Q$ is a {\it characterizing slope} for a knot $L\subset S^3$, if $S^3_{p/q}(K)\cong S^3_{p/q}(L)$ implies $K=L$, where ``$\cong$'' stands for orientation-preserving homeomorphism.
A famous theorem of Kronheimer--Mrowka--Ozsv\'ath--Szab\'o \cite{KMOSz} says that all rational numbers are characterizing slopes for the unknot, confirming a
conjecture of Gordon. Special cases of this theorem were proved in \cite{CGLS,GL,Gabai3}, and a Heegaard Floer proof of this theorem was given in \cite{OSzRatSurg}. Ozsv\'ath and Szab\'o \cite{OSz3141} further proved that all rational numbers are characterizing slopes for the trefoils and the figure-8 knot. These knots are the only knots for which the sets of characterizing slopes are known.

Dehn surgeries on torus knots have been classified in \cite{Moser}. For the torus knot $T_{r,s}$, $r,s>0$, the $rs$--surgery is a connected sum of two lens spaces. If $\frac pq=rs\pm\frac1q$, the $\frac pq$--surgery will be the lens space $L(p,qs^2)$, where the convention for $L(p,q)$ is that it is oriented as the $\frac pq$--surgery on the unknot. For all other slopes, the $\frac pq$--surgery is a Seifert fibered space over $S^2(r,s,|ps-qr|)$.

Ni and Zhang \cite{NiZhang} proved that $\frac pq$ is characterizing for the torus knot $T_{r,s}$ whenever $\frac pq>\frac{30(r^2-1)(s^2-1)}{67}$, where $r,s>0$. McCoy \cite{McCoySharp} improved this lower bound to $\frac{43}4(rs-r-s)$ which is linear in terms of the genus. As a corollary of Theorem~\ref{thm:NotSFS}, we improve the lower bound.

\begin{cor}\label{cor:CharSlope}
Let $r,s>0$ be a pair of relatively prime integers.
Then a slope $\frac pq$ is characterizing for the torus knot $T_{r,s}$ whenever \[\frac pq\ge4g(T_{r,s})+4=2(r-1)(s-1)+4.\]
\end{cor}

Note that Corollary~\ref{cor:CharSlope} contains one previous unknown case of finite surgery: the $13$--surgery on $T_{5,2}$, which is an $I$--type spherical manifold.

\begin{rem}
The bound $4g(T_{r,s})+4$ comes from a theorem of McCoy \cite[Theorem~1.1]{McCoySharp}, which says $g(K)=g(T_{r,s})$ if $S^3_{p/q}(K)\cong S^3_{p/q}(T_{r,s})$ and $\frac pq\ge4g(T_{r,s})+4$. As mentioned in \cite[Remark~3.9]{McCoySharp}, it is possible to lower the bound $4g(T_{r,s})+4$. Thus it is conceivable to lower the bound in Corollary~\ref{cor:CharSlope} to $4g(T_{r,s})$. However, this cannot hold for all torus knots, since $S^3_{21}(T_{5,4})\cong S^3_{21}(T_{11,2})$ \cite[Example~1.1]{NiZhang}, and $21=4g(T_{11,2})+1$. So we will not seek a sharper bound for simplicity. Instead, we will focus on the surgeries with finite $\pi_1$ not covered by Corollary~\ref{cor:CharSlope}.
\end{rem}

\begin{cor}\label{cor:1417}
The slope $14$ is a characterizing slope for $T_{4,3}$, and the slope $17$ is a characterizing slope for $T_{5,3}$.
\end{cor}


A particularly interesting case is $s=2$, $r=2n+1$, where $n=g(T_{r,2})$. The slopes $4n+t$, $t=0,1,2,3,4$, are special. The $(4n+2)$--surgery is reducible. The $(4n+1)$ and $(4n+3)$--surgeries are lens spaces $L(4n+1,4)$ and $L(4n+3,4)$. The $4n$ and $(4n+4)$--surgeries are prism manifolds.
It follows from a theorem of Greene \cite{GreeneCabling} that $4n+2$ is characterizing. Baker \cite{Baker} proved that if $S^3_{p}(K)$ is a lens space with $p\ge4g(K)-1$, then $K$ is a doubly primitive knot. Building on Baker's work, Rasmussen \cite{RasBerge} showed that $4n+3$ is characterizing.
The fact that $4n$ and $4n+4$ are characterizing slopes was proved by Ni--Zhang \cite{NiZhangFinite}. 

Greene \cite{GreeneBerge} proved that the list of doubly primitive knots given by Berge \cite{Berge} is complete. Combined with Baker's work, one can answer the question whether $4n+1$ is a characterizing slope for $T_{2n+1,2}$, but the author did not find such statement in the literature. In this paper, we will provide an answer to this question without using Baker's work.

\begin{thm}\label{thm:4n+1}
If the $(4n+1)$--surgery on a knot $K\subset S^3$ is $L(4n+1,4)$, then either $K=T_{2n+1,2}$ or $n=5$ and $K$ has the same knot Floer homology as $T_{5,4}$. 
\end{thm}

Using a recent result of Tange \cite{Tange}, we also get the following characterization of $T_{2n+1,2}$ among all knots with lens space surgeries.

\begin{cor}\label{cor:Alexander}
If a knot $K\subset S^3$ admits a positive lens space surgery and
\[
\Delta_K(t)=t^n-t^{n-1}+t^{n-2}+\text{lower order terms},
\]
then $K=T_{2n+1,2}$.
\end{cor}
\begin{proof}
Suppose that the $p$--surgery on $K$ is a lens space $L(p,q)$. By \cite{Tange}, $L(p,q)$ is also the $p$--surgery on $T_{2n+1,2}$. Our conclusion follows from Theorem~\ref{thm:4n+1} and \cite{RasBerge}.
\end{proof}

This paper is organized as follows. In Section~\ref{sect:Prelim}, we briefly recall some results about essential laminations and changemaker lattices. In Section~\ref{sect:SFS}, we prove Theorem~\ref{thm:NotSFS} using essential laminations. Corollaries~\ref{cor:CharSlope} and~\ref{cor:1417} are also proved. In Section~\ref{sect:4n+1}, we prove Theorem~\ref{thm:4n+1} using Greene's theory of changemaker lattices. In the appendix, we prove Theorem~\ref{thm:SecondTerm}, which says that if the second term of the knot Floer homology of a fibered knot has rank $1$, then the monodromy is either right-veering or left-veering.

\vspace{5pt}\noindent{\bf Acknowledgements.}\quad  The author was
partially supported by NSF grant numbers DMS-1252992 and DMS-1811900. The author wishes to thank David Gabai for commenting on an earlier draft of this paper, and referring the author to lots of work on laminations. The author is also grateful to Xingru Zhang for informing him the work of Ying-Qing Wu \cite{WuLamination}.



\section{Preliminaries}\label{sect:Prelim}

In this section, we will collect some results about essential laminations and changemaker lattices we will use.

\subsection{Essential laminations}

Essential laminations were introduced by Gabai and Oertel \cite{GO} as a generalization of incompressible surfaces. There are two special classes of essential laminations: very full laminations (defined by Gabai and Mosher, see \cite{Calegari}) and genuine laminations \cite{GK}. A related concept is essential branched surfaces.
We will not give the definitions here. Instead, we will list the necessary results.

The first result is in Gabai and Oertel's original paper \cite[Proposition~4.5]{GO}.


\begin{prop}[Gabai--Oertel]
A lamination $\lambda$ is essential if and only if a splitting of $\lambda$ is fully carried by an essential branched surface.
\end{prop}

The next theorem is due to Brittenham and Claus \cite{Claus}. See \cite[Corollary~4]{Brittenham1} and the paragraph after it. 

\begin{thm}[Brittenham, Claus]\label{thm:Brittenham}
If a Seifert fibered space with base orbifold $S^2(a,b,c)$ contains an essential lamination, then it also contains a taut foliation. 
\end{thm}

Let $K\subset S^3$ be a hyperbolic knot. If $\lambda$ is a very full lamination in the knot complement, one can define the {\it degeneracy locus} $d(\lambda)$ which is in the form $\frac mn$. Here $m,n$ are integers which are not necessarily relatively prime, $(m,n)\ne(0,0)$. Let $c=\gcd(m,n)$, the number $\frac {m/c}{n/c}\in\mathbb Q\cup\{\infty\}$ is called the {\it degeneracy slope}. Gabai and Mosher proved that any hyperbolic knot complement contains a very full lamination.

When $K$ is fibered, the stable lamination transverse to the fibration is very full. In this case, $\frac1{d(\lambda)}\in\mathbb Q$ is also known as the {\it fractional Dehn twist coefficient} of the monodromy \cite{HKM}. The following proposition can be found in \cite[Proposition~3.1]{HKM}.

\begin{prop}[Honda--Kazez--Mati\'c]\label{prop:RV}
If $K$ is a hyperbolic fibered knot, then the monodromy of $K$ is right-veering if and only if $\frac1{d(\lambda)}>0$.
\end{prop}

As a result, the monodromy of $K$ is left-veering if and only if $\frac1{d(\lambda)}<0$. If the monodromy is neither right-veering nor left-veering, then $d(\lambda)=\frac m0$.

The following theorem was due to Gabai \cite[Theorem~8.8]{GabaiProblems}.

\begin{thm}[Gabai]\label{thm:Bound}
Let $K\subset S^3$ be a hyperbolic knot, $\lambda$ be a very full lamination in its complement, then the degeneracy locus $d(\lambda)$ is either $\frac m0$ or $\frac m1$ for an integer $m$ with $|m|\le 4g(K)-2$. If $K$ is fibered,
$\lambda$ is the stable lamination transverse to the fibration, and $d(\lambda)=\frac m1$, then $|m|\ge2$.
\end{thm}

If $\frac pq$ is a slope, and $d(\lambda)=\frac mn$, define $\Delta(d(\lambda),\frac pq)=|pn-qm|$. The first part of the following theorem can be found in \cite[Theorem~5.3]{GO}, and the same proof yields the second part. Gabai was aware of the result about genuine laminations which was explicitly stated by Brittenham \cite{Brittenham3}. Gabai and Mosher also proved this theorem for any hyperbolic cusped manifolds, see
\cite[Theorem~6.48]{Calegari}.

\begin{thm}[Gabai]\label{thm:Lamination}
Let $K\subset S^3$ be a hyperbolic fibered knot, $\lambda$ be the stable lamination transverse to the fibration.
If $\Delta(d(\lambda),\frac pq)\ge2$, $\lambda$ will be an essential lamination in $S^3_{p/q}(K)$. If $\Delta(d(\lambda),\frac pq)\ge3$, $\lambda$ will be a genuine lamination in $S^3_{p/q}(K)$.
\end{thm} 

When a closed atoroidal manifold $Y$ contains a genuine lamination, Gabai and Kazez \cite{GK} proved that $\pi_1(Y)$ is word-hyperbolic. For our purpose, we only need the following theorem of Wu (see \cite[Theorem~2.5]{WuLamination} and the paragraph preceding it).

\begin{thm}[Wu]\label{thm:Wu}
Suppose that $K\subset S^3$ is a hyperbolic knot. Let $B$ be an essential branched surface in $S^3\setminus K$, and let $\gamma_0$ be a degeneracy slope corresponding to $B$. If $\Delta(\gamma_0,\frac pq)>2$, then $S^3_{\frac pq}(K)$ is hyperbolic.
\end{thm}

\begin{rem}
In \cite{WuLamination}, it is stated that $S^3_{\frac pq}(K)$ is ``hyperbolike'', namely, it is irreducible, atoroidal, and is not a Seifert fibered space. However, by Perelman's resolution of the Geometrization Conjecture \cite{P1,P2,P3}, we know such manifolds are hyperbolic.
\end{rem}

\subsection{Changemaker lattices}

In this subsection, we briefly recall the theory of changemaker lattices of Greene \cite{GreeneBerge,GreeneCabling}.

When $Y$ is a rational homology sphere and $\mathfrak t\in\spinc(Y)$, Ozsv\'ath and Szab\'o \cite{OSzAbGr} defined the {\it correction term} $d(Y,\mathfrak t)\in\mathbb Q$ which is an invariant of the pair $(Y,\mathfrak t)$.
If $Y$ is the boundary of a smooth, compact, negative definite $4$--manifold $X$, then
\begin{equation}\label{eq:CorrBound}
 c_1(\mathfrak s)^2 + b_2(X)\le 4d(Y, \mathfrak t),
\end{equation}
for any $\mathfrak s \in \spinc(X)$ that extends $\mathfrak t \in \text{Spin}^c(Y)$.

Suppose that $Y$ is obtained by $p$--surgery on a knot $K\subset S^3$, $p>0$, then $Y$ is the boundary of a $4$--manifold $W=W_p(K)$ which consists of a zero-handle and a two-handle with attaching curve $K$. If $Y$ is also the boundary of a smooth, compact, negative definite, simply connected $4$--manifold $X$ with $b_2(X)=n$, then
the four-manifold $Z := X \cup_Y (-W)$ is a smooth, closed, negative definite, simply connected $4$--manifold with $b_2(Z)=n+1$. (The simple connectedness is not necessary, but it suffices for our purpose.) Donaldson's Diagonalization Theorem \cite{Donaldson1} implies that the intersection pairing on $H_2(Z)$ is isomorphic to $-\mathbb Z^{n+1}$, negative of the standard $(n+1)$-dimensional Euclidean integer lattice. Consequently, negative of the intersection pairing on $X$, denoted $\Lambda$, embeds as a codimension one sub-lattice of $\mathbb Z^{n+1}$.

\begin{defn}
A smooth, compact, negative definite $4$--manifold $X$ is {\it sharp} if for every $\mathfrak t \in \text{Spin}^c(Y)$, there exists some $\mathfrak s\in \spinc(X)$ extending $\mathfrak t$ such that the equality is realized in Equation (\ref{eq:CorrBound}).
\end{defn}

\begin{defn}\label{defn:changemaker}
A vector $\sigma=(\sigma_0,\sigma_1,\dots,\sigma_{n})\in\mathbb Z^{n+1}$ that satisfies $0\le\sigma_0\le\sigma_1\le\cdots\le\sigma_{n}$ is a {\it changemaker vector} if for every $k$, with $0\le k\le\sigma_0+\sigma_1+\cdots+\sigma_{n}$, there exists a subset $S\subset\{0,1,\dots,n\}$
such that $k=\sum_{i\in S}\sigma_i$. A sublattice of $\mathbb Z^{n+1}$ is a {\it changemaker lattice} if it is the orthogonal complement of a changemaker vector.
\end{defn}

The following important theorem was proved by Greene \cite{GreeneCabling}.

\begin{thm}[Greene]\label{thm:Greene}
Suppose that $K\subset S^3$ is an L-space knot, and $Y=S^3_p(K)$ is the boundary of a simply connected sharp $4$--manifold $X$ with $b_2(X)=n$. Let $Z=X\cup_{Y}(-W_p(K))$, and let $Q_Z$ be the intersection form on $Z$. Then there is a lattice isomorphism $(H_2(Z),-Q_Z)\to \mathbb Z^{n+1}$, such that the generator of $H_2(-W_p(K))$ is mapped to a changemaker vector $\sigma$ with $\langle\sigma,\sigma\rangle=p$, and
$H_2(X)$ is mapped to the orthogonal complement $(\sigma)^{\perp}$ of $\sigma$. 
\end{thm}

One can recover the normalized Alexander polynomial $\Delta_K(t)=\sum_{i}a_it^i$ from the changemaker vector \cite[Lemma~2.5]{GreeneCabling}.

\begin{lem}\label{lem:Alex}
The torsion coefficients of $K$ can be determined by
\[
t_i(K)=
\left\{
\begin{array}{cl}
\displaystyle\min_{\mathfrak c}\frac{\mathfrak c^2-n-1}8, &\text{for each $i\in\{0,1,\dots,\lfloor\frac p2\rfloor\}$,}\\
&\\
0,&\text{for $i>\frac p2$.}
\end{array}
\right.
\]
where $\mathfrak c$ is subject to
\[
\mathfrak c\in(1,1,\dots,1)+2\mathbb Z^{n+1},\quad\langle\mathfrak c,\sigma\rangle+p\equiv2i\pmod{2p}.
\]
The coefficients $a_i$ of $\Delta_K$ can be determined by the following rule.
For $i>0$,
\[
a_i=t_{i-1}-2t_i+t_{i+1},
\]
and
\[a_0=1-2\sum_{i>0}a_i.\]
\end{lem}

Although the genus of $K$ can be deduced from $\Delta_K$, it is useful to know the following direct formula \cite[Proposition~3.1]{GreeneCabling}:
\begin{equation}\label{eq:Genus}
2g(K)=p-|\sigma|_1=p-\sum_{j=0}^n|\sigma_j|.
\end{equation}


\section{Exceptional surgeries}\label{sect:SFS}

Now we are ready to prove our main theorem.

\begin{proof}[Proof of Theorem~\ref{thm:NotSFS}]
(1) If the monodromy of $K$ is right-veering \cite{HKM}, let $\lambda$ be the stable lamination transverse to the fibration, then the degeneracy locus is positive by Proposition~\ref{prop:RV}. So $d(\lambda)$ has the form $\frac m1$ for a positive integer $m$ with $2\le m\le 4g(K)-2$ by Theorem~\ref{thm:Bound}.
If $\frac pq>4g(K)$ or $\frac pq<0$, $\Delta(\frac pq,\frac m1)\ge3$, so Theorem~\ref{thm:Wu} implies that $S^3_{p/q}(K)$ is hyperbolic, a contradiction.

Theorem~\ref{thm:Lamination} implies that $S^3_{4g(K)}(K)$ contains an essential lamination. If $S^3_{4g(K)}(K)$ is a small Seifert fibered space, 
Theorem~\ref{thm:Brittenham} implies that it has a taut foliation, thus it cannot be an L-space \cite{OSzGenus,KR}. 

\vspace{5pt}\noindent(2) If the monodromy of $K$ is neither right-veering nor left-veering, then $d(\lambda)$ has the form $\frac m0$ for a positive integer $m$ with $1\le m\le 4g(K)-2$ by Theorem~\ref{thm:Bound}. If $|q|\ge3$, $\Delta(\frac pq,\frac 10)=|q|\ge3$, so Theorem~\ref{thm:Wu} implies that $S^3_{p/q}(K)$ is hyperbolic, a contradiction.
\end{proof}

\begin{proof}[Proof of Corollary~\ref{cor:CharSlope}]
When $\frac pq\ge4g(T_{r,s})+4$ and $S^3_{p/q}(K)\cong S^3_{p/q}(T_{r,s})$, we know that $S^3_{p/q}(T_{r,s})$ is an L-space and $K$ is fibered \cite{OSzLens,Gh,NiFibred}. By the work of McCoy \cite[Theorem~1.1]{McCoySharp}, $g(K)=g(T_{r,s})$ and $\Delta_K=\Delta_{T_{r,s}}$. 

If $K$ is hyperbolic, our result follows from Theorem~\ref{thm:NotSFS}.

If $K$ is a torus knot, since $\Delta_K=\Delta_{T_{r,s}}$, we must have $K=T_{r,s}$.

If $K$ is a satellite knot, let $R\subset S^3\setminus K$ be an ``innermost'' incompressible torus. Let $V$ be the solid torus bounded by $R$ in $S^3$, and let $L$ be the core of $V$. Since $S^3_{p/q}(K)$ is irreducible and atoroidal, using \cite{GabaiSolidTori}, $V_{p/q}(K)$ must be a solid torus. In this case, $K$ is a $0$--bridge or $1$--bridge braid in $V$ with winding number $w>1$, and $\gcd(p,w)=1$ for homological reasons. The simple loop with slope $p/(qw^2)$ on $R$ is null-homologous in $V_{p/q}(K)$, so $S^3_{p/q}(K)=S^3_{p/(qw^2)}(L)$. Since $R$ is innermost, $L$ is not a satellite knot, so $L$ is either hyperbolic or a torus knot.

If $L$ is hypebolic, it follows from \cite[Theorem~1.2]{LM} that $w^2\le8$, so $w=2$. Hence $K$ is a $(2h+1,2)$--cable of $L$ and $p/q=4h+2+1/n$ for some $n\in\mathbb Z\setminus\{0\}$. 
Since $\frac pq\ge4g(K)+4$, we have
\[4h+2+\frac1{n}\ge 4g(K)+4=4(2g(L)+h)+4,\]
hence $2+\frac1{n}\ge8g(L)+4$,
which is not possible.

If $L$ is a torus knot, we can get a contradiction by the same argument as in the proof of \cite[Proposition~2.5]{NiZhang}.
\end{proof}

\begin{proof}[Proof of Corollary~\ref{cor:1417}]
Let $p=14$ or $17$, and $T$ be $T_{4,3}$ or $T_{5,3}$. It follows from \cite{Gu} that if $S^3_p(K)\cong S^3_p(T)$, then $\Delta_K=\Delta_T$. See also \cite[Known Facts~1.2~(3)]{NiZhangFinite}. Hence Theorem~\ref{thm:NotSFS} rules out the possibility that $K$ is hyperbolic. The finite surgeries on satellite knots have been classified in \cite[Corollary~1.4]{BZ} and \cite[Theorem~7]{BH}, thus we can check $K$ cannot be satellite.
So $K$ is a torus knot. Since $\Delta_K=\Delta_{T}$, we must have $K=T$.
\end{proof}


\section{The $(4n+1)$--surgery on $T_{2n+1,2}$}\label{sect:4n+1}

Throughout this section, let $T=T_{2n+1,2}$ and let $K$ be a knot satisfying $S^3_{4n+1}(K)\cong S^3_{4n+1}(T)$.

\begin{figure}[t]
\begin{center}
\includegraphics[scale=.4]{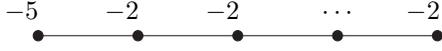}
\put(-166,8.5){$-5$}
\put(-128,8.5){$-2$}
\put(-90,8.5){$-2$}
\put(-46,8.5){$\cdots$}
\put(-14.135,8.5){$-2$}
\caption{\label{fig:Linear}A negative definite plumbing diagram of $X$. There are $n-1$ copies of $-2$.}
\end{center}
\end{figure}

The following lemma is a small part of Greene's lens space realization theorem \cite{GreeneBerge}. Since we do not need the full strength of Greene's theorem, we will give a quick proof here.

\begin{lem}\label{lem:SameAlex}
Either $\Delta_K=\Delta_T$, or $n=5$ and $\Delta_K=\Delta_{T_{5,4}}$.
\end{lem}
\begin{proof}
We know $Y=S^3_{4n+1}(T)\cong L(4n+1,4)=-L(4n+1,n)$. (Note that our convention of the orientation of $L(p,q)$ is opposite to the convention in \cite{GreeneBerge}.) Since
\[
\frac {4n+1}{n}=[5,2,\dots,2]^-:=5-\frac1{2-\displaystyle\frac1{2-\displaystyle\frac1{\ddots-\displaystyle\frac1{2}}}},
\]
where there are $(n-1)$ copies of $2$ in the expression,
we conclude that $Y$ is the boundary of a negatively plumbed $4$--manifold $X$ with plumbing diagram given in Figure~\ref{fig:Linear}.
Since $X$ is sharp, it follows from Theorem~\ref{thm:Greene} that the lattice $\Lambda$ is a changemaker lattice.

Now $\Lambda$ has a basis consisting of one vector  with norm $5$ and $(n-1)$ vectors with norm $2$. Any norm $2$ vector in $\mathbb Z^{n+1}$ must be of the form $e_i\pm e_j$, where $e_0,e_1,\dots,e_n$ is the usual orthonormal basis. Since $\Lambda=(\sigma)^{\perp}$ for a changemaker vector $\sigma$, the two coordinates of $\sigma$ corresponding to $e_i$ and $e_j$ must be equal. Thus we conclude that $n$ coordinates of $\sigma$ are equal. Since $\sigma_0=1$ (see the proof of \cite[Theorem~1.6]{GreeneBerge}), $\sigma$ must be of the form
\[
(1,a,a,\dots,a) \text{ or } (1,1,1,\dots,1,b)
\]
for some $a,b\ge1$. Since $\Lambda=(\sigma)^{\perp}$ contains a vector with norm $5$, we have $a=2$ and $b\in\{2,4\}$. 

If $\sigma=(1,2,2,\dots,2)$, using Lemma~\ref{lem:Alex}, we can get $\Delta_K=\Delta_T$. (It is easier to see $g(K)=n$ using (\ref{eq:Genus}).)

If $b=2$, the norm of $\sigma$ is $n+4$. By Theorem~\ref{thm:Greene}, the norm of $\sigma$ is $4n+1$, so $n=1$. This falls into the case we just considered.

If $b=4$, the norm of $\sigma$ is $n+16$. Since the norm of $\sigma$ is $4n+1$, $n=5$. By the computation in \cite[Section~10.3]{OSzAbGr}, $\Delta_K=\Delta_{T_{5,4}}$.
\end{proof}

\begin{proof}[Proof of Theorem~\ref{thm:4n+1}]
By Lemma~\ref{lem:SameAlex}, either $\Delta_K=\Delta_T$, or $n=5$ and $\Delta_K=\Delta_{T_{5,4}}$.
In the latter case, $K$ has the same knot Floer homology as $T_{5,4}$ by \cite{OSzLens}.

Now we consider the case $\Delta_K=\Delta_T$. We have $g(K)=g(T)$.

If $K$ is hyperbolic, we get a contradiction by Theorem~\ref{thm:NotSFS}.

If $K$ is a torus knot, since $\Delta_K=\Delta_T$, we get $K=T$.

If $K$ is a satellite knot,
using the classification of
lens space surgeries on satellite knots \cite{Wang,BL,Wu}, we see that $K$ is the $(2uv+1,2)$--cable of $T_{u,v}$ and $4n+1=4uv+1$. Then 
\[
g(K)=(u-1)(v-1)+uv>uv=n,
\] 
a contradiction.
\end{proof}


\setcounter{section}{1}
\setcounter{thm}{0}
\renewcommand{\thesection}{\Alph{section}}

\section*{Appendix: A criterion for veering}

In this appendix, we assume the readers are familiar with the basic notions of knot Floer homology \cite{OSzKnot,RasThesis}.
We will prove the following theorem.

\begin{thm}\label{thm:SecondTerm}
Let $Y$ be a closed, oriented $3$--manifold, $K\subset Y$ be a fibered knot with fiber $F$ and monodromy $\phi$. If 
\[\mathrm{rank}\widehat{HFK}(Y,K,[F],g(F)-1)=1,\]
then $\phi$ is either right-veering or left-veering.
\end{thm}

Theorem~\ref{thm:SecondTerm} easily follows from Baldwin and Vela-Vick's work \cite{BVV}, which will be introduced below.

Let $K$ be a null-homologous knot in a closed, oriented $3$--manifold $Y$, and let $F$ be a Seifert surface.
There is a chain map \[\partial_z\co\widehat{CFK}(Y,K,[F],i)\to\widehat{CFK}(Y,K,[F],i-1)\] defined by counting holomorphic disks with $n_z=1$. Similarly, there is a chain map 
\[\partial_w\co\widehat{CFK}(Y,K,[F],i-1)\to\widehat{CFK}(Y,K,[F],i).\]

The following theorem is contained in the proof of \cite[Theorem~1.1]{BVV}.

\begin{thm}\label{thm:BVV}
Let $K$ be a fibered knot in a closed, oriented $3$--manifold $Y$. Let $F$ be a Seifert surface, and let $\phi\co F\to F$ be the monodromy of the corresponding open book.
If $\phi$ is not left-veering, then the induced map
\[
(\partial_z)_*\co \widehat{HFK}(Y,K,[F],1-g(F))\to \widehat{HFK}(Y,K,[F],-g(F))
\]
is nonzero.  Similarly, if  $\phi$ is not right-veering, then the induced map
\[
(\partial_w)_*\co \widehat{HFK}(Y,K,[F],-g(F))\to \widehat{HFK}(Y,K,[F],1-g(F))
\]
is nonzero.
\end{thm}

\begin{proof}[Proof of Theorem~\ref{thm:SecondTerm}]
We will use $\mathbb Q$--coefficients for Heegaard Floer homology.
Assume that $\phi$ is neither right-veering nor left-veering.
Let $C=CFK^{\infty}(Y,K,[F])$ be the $\mathbb Z^2$--filtered knot Floer chain complex. By
\cite[Lemma~4.5]{RasThesis}, $C$ is filtered chain homotopy equivalent to a chain complex $C'$ with 
\[
C'\{(i,j)\}\cong \widehat{HFK}(Y,K,[F],j-i).
\]
Let $a$ be a generator of $C'\{(0,-g(F))\}$, consider the component $b$ of  $\partial ^2a$ in $C'\{(-1,-1-g(F))\}$. On one hand, since $\partial ^2=0$, $b=0$. On the other hand, $b$ is just $(\partial_z)_*\circ(\partial_w)_*(a)$. 
By Theorem~\ref{thm:BVV}, both $(\partial_z)_*$ and $(\partial_w)_*$ are isomorphisms, so $b\ne0$, a contradiction.\end{proof}

It is well-known that $\widehat{HFK}(Y,K,[F],g(F))$, the topmost term in knot Floer homology, contains a lot of information about the topology of the knot complement \cite{OSzGenus,Gh,NiFibred}. It is natural to ask what topological information is contained in other terms of $\widehat{HFK}(Y,K)$. Baldwin and Vela-Vick's work \cite{BVV} and our Theorem~\ref{thm:SecondTerm} gave some partial answers to this question. 


\begin{thebibliography}{H}

\bib{Baker}{article}{
   author={Baker, Kenneth L.},
   title={Small genus knots in lens spaces have small bridge number},
   journal={Algebr. Geom. Topol.},
   volume={6},
   date={2006},
   pages={1519--1621},
   issn={1472-2747},
}

\bib{BVV}{article}{
   author={Baldwin, John},
   author={Vela-Vick, David Shea},
   title={A note on the knot Floer homology of fibered knots},
   journal={Algebr. Geom. Topol.},
   volume={18},
   date={2018},
   number={6},
   pages={3669--3690},
   issn={1472-2747},
}

\bib{Berge}{article}{
   author={Berge, John},
   title={Some knots with surgeries yielding lens spaces},
   eprint={https://arxiv.org/abs/1802.09722},
status={preprint},
     date={1990}
}

\bib{BH}{article}{
   author={Bleiler, Steven A.},
   author={Hodgson, Craig D.},
   title={Spherical space forms and Dehn filling},
   journal={Topology},
   volume={35},
   date={1996},
   number={3},
   pages={809--833},
   issn={0040-9383},
}

\bib{BL}{article}{
   author={Bleiler, Steven A.},
   author={Litherland, Richard A.},
   title={Lens spaces and Dehn surgery},
   journal={Proc. Amer. Math. Soc.},
   volume={107},
   date={1989},
   number={4},
   pages={1127--1131},
   issn={0002-9939},
}

\bib{BZ}{article}{
   author={Boyer, S.},
   author={Zhang, X.},
   title={Finite Dehn surgery on knots},
   journal={J. Amer. Math. Soc.},
   volume={9},
   date={1996},
   number={4},
   pages={1005--1050},
   issn={0894-0347},
}
	

\bib{Brittenham1}{article}{
   author={Brittenham, Mark},
   title={Essential laminations in Seifert-fibered spaces},
   journal={Topology},
   volume={32},
   date={1993},
   number={1},
   pages={61--85},
}


\bib{Brittenham3}{article}{
   author={Brittenham, Mark},
   title={Essential laminations, exceptional Seifert-fibered spaces, and
   Dehn filling},
   journal={J. Knot Theory Ramifications},
   volume={7},
   date={1998},
   number={4},
   pages={425--432},
   issn={0218-2165},
}

\bib{Calegari}{book}{
   author={Calegari, Danny},
   title={Foliations and the geometry of 3-manifolds},
   series={Oxford Mathematical Monographs},
   publisher={Oxford University Press, Oxford},
   date={2007},
   pages={xiv+363},
}

\bib{Claus}{book}{
   author={Claus, Wilhelmina Christina},
   title={Essential laminations in closed Seifert fibered spaces},
   note={Thesis (Ph.D.)--The University of Texas at Austin},
   publisher={ProQuest LLC, Ann Arbor, MI},
   date={1991},
   pages={109},
}

\bib{CGLS}{article}{
   author={Culler, Marc},
   author={Gordon, C. McA.},
   author={Luecke, J.},
   author={Shalen, Peter B.},
   title={Dehn surgery on knots},
   journal={Ann. of Math. (2)},
   volume={125},
   date={1987},
   number={2},
   pages={237--300},
}

\bib{Donaldson1}{article}{
   author={Donaldson, S. K.},
   title={An application of gauge theory to four-dimensional topology},
   journal={J. Differential Geom.},
   volume={18},
   date={1983},
   number={2},
   pages={279--315},
}


\bib{Gabai3}{article}{
   author={Gabai, David},
   title={Foliations and the topology of $3$-manifolds. III},
   journal={J. Differential Geom.},
   volume={26},
   date={1987},
   number={3},
   pages={479--536},
}

\bib{GabaiSolidTori}{article}{
   author={Gabai, David},
   title={Surgery on knots in solid tori},
   journal={Topology},
   volume={28},
   date={1989},
   number={1},
   pages={1--6},
   issn={0040-9383},
}

\bib{GabaiProblems}{article}{
   author={Gabai, David},
   title={Problems in foliations and laminations},
   conference={
      title={Geometric topology},
      address={Athens, GA},
      date={1993},
   },
   book={
      series={AMS/IP Stud. Adv. Math.},
      volume={2},
      publisher={Amer. Math. Soc., Providence, RI},
   },
   date={1997},
   pages={1--33},
}


\bib{GK}{article}{
   author={Gabai, David},
   author={Kazez, William H.},
   title={Group negative curvature for $3$-manifolds with genuine
   laminations},
   journal={Geom. Topol.},
   volume={2},
   date={1998},
   pages={65--77},
}

\bib{GO}{article}{
   author={Gabai, David},
   author={Oertel, Ulrich},
   title={Essential laminations in $3$-manifolds},
   journal={Ann. of Math. (2)},
   volume={130},
   date={1989},
   number={1},
   pages={41--73},
}

\bib{Gh}{article}{
   author={Ghiggini, Paolo},
   title={Knot Floer homology detects genus-one fibred knots},
   journal={Amer. J. Math.},
   volume={130},
   date={2008},
   number={5},
   pages={1151--1169},
}


\bib{GL}{article}{
   author={Gordon, C. McA.},
   author={Luecke, J.},
   title={Knots are determined by their complements},
   journal={J. Amer. Math. Soc.},
   volume={2},
   date={1989},
   number={2},
   pages={371--415},
}


\bib{GreeneBerge}{article}{
   author={Greene, Joshua Evan},
   title={The lens space realization problem},
   journal={Ann. of Math. (2)},
   volume={177},
   date={2013},
   number={2},
   pages={449--511},
   issn={0003-486X},
}

\bib{GreeneCabling}{article}{
   author={Greene, Joshua Evan},
   title={L-space surgeries, genus bounds, and the cabling conjecture},
   journal={J. Differential Geom.},
   volume={100},
   date={2015},
   number={3},
   pages={491--506},
   issn={0022-040X},
}

\bib{Gu}{article}{
   author={Gu, Liling},
   title={Integral finite surgeries on knots in $S^3$},
   status={preprint},
   date={2014},
   eprint={https://arxiv.org/abs/1401.6708}
}

\bib{HKM}{article}{
   author={Honda, Ko},
   author={Kazez, William H.},
   author={Mati\'{c}, Gordana},
   title={Right-veering diffeomorphisms of compact surfaces with boundary},
   journal={Invent. Math.},
   volume={169},
   date={2007},
   number={2},
   pages={427--449},
   issn={0020-9910},
}

\bib{KR}{article}{
   author={Kazez, William H.},
   author={Roberts, Rachel},
   title={$C^0$ approximations of foliations},
   journal={Geom. Topol.},
   volume={21},
   date={2017},
   number={6},
   pages={3601--3657},
}

\bib{KMOSz}{article}{
   author={Kronheimer, P.},
   author={Mrowka, T.},
   author={Ozsv\'{a}th, P.},
   author={Szab\'{o}, Z.},
   title={Monopoles and lens space surgeries},
   journal={Ann. of Math. (2)},
   volume={165},
   date={2007},
   number={2},
   pages={457--546},
}


\bib{LM}{article}{
   author={Lackenby, Marc},
   author={Meyerhoff, Robert},
   title={The maximal number of exceptional Dehn surgeries},
   journal={Invent. Math.},
   volume={191},
   date={2013},
   number={2},
   pages={341--382},
   issn={0020-9910},
}


\bib{McCoySharp}{article}{
   author={McCoy, D.},
   title={Surgeries, sharp 4-manifolds and the Alexander polynomial},
   status={to appear Algebr. Geom. Topol.},
   date={2014},
   eprint={https://arxiv.org/abs/1412.0572}
}

\bib{McCoyTorus}{article}{
   author={McCoy, Duncan},
   title={Non-integer characterizing slopes for torus knots},
   journal={Comm. Anal. Geom.},
   volume={28},
   date={2020},
   number={7},
   pages={1647--1682},
   issn={1019-8385},
}

\bib{Moser}{article}{
   author={Moser, Louise},
   title={Elementary surgery along a torus knot},
   journal={Pacific J. Math.},
   volume={38},
   date={1971},
   pages={737--745},
}

\bib{NiFibred}{article}{
   author={Ni, Yi},
   title={Knot Floer homology detects fibred knots},
   journal={Invent. Math.},
   volume={170},
   date={2007},
   number={3},
   pages={577--608},
}

\bib{NiZhang}{article}{
   author={Ni, Yi},
   author={Zhang, Xingru},
   title={Characterizing slopes for torus knots},
   journal={Algebr. Geom. Topol.},
   volume={14},
   date={2014},
   number={3},
   pages={1249--1274},
}

\bib{NiZhangFinite}{article}{
   author={Ni, Yi},
   author={Zhang, Xingru},
   title={Finite Dehn surgeries on knots in $S^3$},
   journal={Algebr. Geom. Topol.},
   volume={18},
   date={2018},
   number={1},
   pages={441--492},
   issn={1472-2747},
}

\bib{OSzAbGr}{article}{
   author={Ozsv\'ath, Peter},
   author={Szab\'o, Zolt\'an},
   title={Absolutely graded Floer homologies and intersection forms for
   four-manifolds with boundary},
   journal={Adv. Math.},
   volume={173},
   date={2003},
   number={2},
   pages={179--261},
   issn={0001-8708},
}

\bib{OSzKnot}{article}{
   author={Ozsv\'ath, Peter},
   author={Szab\'o, Zolt\'an},
   title={Holomorphic disks and knot invariants},
   journal={Adv. Math.},
   volume={186},
   date={2004},
   number={1},
   pages={58--116},
   issn={0001-8708},
}

\bib{OSzGenus}{article}{
   author={Ozsv\'ath, Peter},
   author={Szab\'o, Zolt\'an},
   title={Holomorphic disks and genus bounds},
   journal={Geom. Topol.},
   volume={8},
   date={2004},
   pages={311--334},
   issn={1465-3060},
}

\bib{OSzLens}{article}{
   author={Ozsv\'ath, Peter},
   author={Szab\'o, Zolt\'an},
   title={On knot Floer homology and lens space surgeries},
   journal={Topology},
   volume={44},
   date={2005},
   number={6},
   pages={1281--1300}
}

\bib{OSzContact}{article}{
   author={Ozsv\'ath, Peter},
   author={Szab\'o, Zolt\'an},
   title={Heegaard Floer homology and contact structures},
   journal={Duke Math. J.},
   volume={129},
   date={2005},
   number={1},
   pages={39--61},
   issn={0012-7094},
}

\bib{OSzRatSurg}{article}{
   author={Ozsv\'ath, Peter},
   author={Szab\'o, Zolt\'an},
   title={Knot Floer homology and rational surgeries},
   journal={Algebr. Geom. Topol.},
   volume={11},
   date={2011},
   number={1},
   pages={1--68},
   issn={1472-2747},
}

\bib{OSz3141}{article}{
   author={Ozsv\'ath, Peter},
   author={Szab\'o, Zolt\'an},
   title={The Dehn surgery characterization of the trefoil and the figure
   eight knot},
   journal={J. Symplectic Geom.},
   volume={17},
   date={2019},
   number={1},
   pages={251--265},
}

\bib{P1}{article}{
   author={Perelman, Grigori},
   title={The entropy formula for the Ricci flow and its geometric applications},
   status={preprint},
   date={2002},
   eprint={https://arxiv.org/abs/math/0211159}
}

\bib{P2}{article}{
   author={Perelman, Grigori},
   title={Ricci flow with surgery on three-manifolds},
   status={preprint},
   date={2003},
   eprint={https://arxiv.org/abs/math/0303109}
}

\bib{P3}{article}{
   author={Perelman, Grigori},
   title={Finite extinction time for the solutions to the Ricci flow on certain three-manifolds},
   status={preprint},
   date={2003},
   eprint={https://arxiv.org/abs/math/0307245}
}

\bib{RasThesis}{book}{
   author={Rasmussen, Jacob Andrew},
   title={Floer homology and knot complements},
   note={Thesis (Ph.D.)--Harvard University},
   publisher={ProQuest LLC, Ann Arbor, MI},
   date={2003},
   pages={126},
}


\bib{RasBerge}{article}{
   author={Rasmussen, Jacob Andrew},
   title={Lens space surgeries and L-space homology spheres},
   status={preprint},
   date={2007},
   eprint={https://arxiv.org/abs/0710.2531}
}

\bib{Tange}{article}{
   author={Tange, Motoo},
   title={The third term in lens surgery polynomials},
   journal={Hiroshima Math. J.},
   volume={51},
   date={2021},
   number={1},
   pages={101--109},
}


\bib{Wang}{article}{
   author={Wang, Shicheng},
   title={Cyclic surgery on knots},
   journal={Proc. Amer. Math. Soc.},
   volume={107},
   date={1989},
   number={4},
   pages={1091--1094},
   issn={0002-9939},
}

\bib{Wu}{article}{
   author={Wu, Ying-Qing},
   title={Cyclic surgery and satellite knots},
   journal={Topology Appl.},
   volume={36},
   date={1990},
   number={3},
   pages={205--208},
   issn={0166-8641},
}

\bib{WuLamination}{article}{
   author={Wu, Ying-Qing},
   title={Sutured manifold hierarchies, essential laminations, and Dehn
   surgery},
   journal={J. Differential Geom.},
   volume={48},
   date={1998},
   number={3},
   pages={407--437},
}

\end{thebibliography}
\end{document}